\documentclass[11pt,reqno]{amsart}

\usepackage{amsfonts}
\usepackage{amsmath}
\usepackage{amssymb, esint}
\usepackage{amsthm,color}
\usepackage{enumerate}
\usepackage{mathtools}
\usepackage{enumitem}
\usepackage{url}
\usepackage[hidelinks, bookmarksdepth=3]{hyperref}
\usepackage{xcolor}
\usepackage{cancel}
\usepackage{ulem}
\usepackage{caption,subcaption}
\usepackage{tikz}
\usepackage{tikz-3dplot}

\usepackage{longtable}

\usepackage{soul}

\tdplotsetmaincoords{70}{110}
\usepackage{geometry}
\newcommand{\R}{{\mathbb R}}

\def\0{{\mathbf 0}}

\newcounter{textblock}

\usepackage{orcidlink}

\theoremstyle{plain}
\newtheorem{thm}{Theorem}[section]

\newtheorem{cor}[thm]{Corollary}
\newtheorem{lem}[thm]{Lemma}

\newcommand{\thistheoremnames}{}
\newtheorem*{genericthms}{\thistheoremnames}
\newenvironment{para*}[1]
  {\renewcommand{\thistheoremnames}{#1}%
   \begin{genericthms}}
  {\end{genericthms}}

\theoremstyle{remark}

\newtheorem*{claim*}{Claim}

\newtheorem{rem}[thm]{Remark}

\numberwithin{equation}{section}


\title[Partial regularity of the  gradient for subsolutions]{Partial regularity of the  gradient for subsolutions}

\author{A. Hakobyan}\address[Aram Hakobyan]{Yerevan, Armenia.}
\email{hakobyan.aram@gmail.com}

\author{M. Poghosyan \, \orcidlink{0009-0002-0066-5494}} 
\address[Michael Poghosyan] {Department of Mathematics,  Yerevan State University, Armenia.}
\email{michael@ysu.am}

\author[Shahgholian]{Henrik Shahgholian\,\orcidlink{0000-0002-1316-7913}}
\address{Department of Mathematics, KTH Royal Institute of Technology, SE-10044 Stockholm, Sweden}
\email{henriksh@kth.se
}

\begin{document}

\begin{abstract}
We prove that the gradient of any bounded subharmonic function is upper semi-continuous, provided that its super-level sets can be touched from the exterior by  uniform $C^{1,\text{Dini}}$ domains at every point.
This idea extends to a class of  general operators, as well as  to  the boundary behaviour of the gradient of  solutions of the Dirichlet problem  in a domain whose boundary satisfy this geometric condition.

\end{abstract}

\maketitle

\section{Rationale}

In the unit ball $B_1 \subset \R^n$ ($n\geq 2$), let $u$ be a bounded subharmonic  function,  
\begin{equation}\label{eq:sub}
\Delta u \geq 0 ,
\end{equation}
in the weak sense, meaning $u \in W^{1,2} (B_1)$
$$
\int \nabla u \cdot \nabla \varphi \leq 0, 
\qquad \forall  \quad
0 \leq \varphi \in W_0^{1,2}(B_1).
$$

The question we raise is: 

    \begin{center}
    {\it Find minimal geometric conditions on the level sets of  $u$ in $B_{1/2}$, \\ such that $\nabla u$ is upper semi-continuous.  }
\end{center}

To  our  knowledge,  the question of the upper semi-continuity of the gradient at the boundary appeared first\footnote{This is probably  the only place where such a question is studied! However several recent works have considered the question of differentiability of $u$ up to the boundary for Dirichlet problem.}  in one of the  earlier articles by the 3rd author \cite{Henrot-Sh-2000}, which studied the existence of solutions to the Bernoulli free boundary problem in the convex setting. In that work, proving the  upper semi-continuity of the gradient of the capacitor potential (in convex rings) was a technical ingredient in the argument.

However, we later observed that this property can be viewed in a broader framework, prompting us to examine the boundary behaviour of the gradient of solutions to PDEs under geometric conditions on the boundary beyond convexity. As we pursued a more systematic understanding of the problem, it became clear that this idea  applies in a much greater generality than initially expected, which ultimately motivated the present paper.

Let us first mention two potential objections that might be raised to such a seemingly ad hoc question, especially when posed at this level of generality:

\medskip 

(i) Why do we consider upper rather than lower semi-continuity?

(ii) Why not simply redefine the gradient as its upper semi-continuous envelope?

\medskip

Indeed, if   the level surface $\partial\{u > u(x^0)\}$ is $C^1$ in a neighborhood of $x^0$, then (formally) one has\footnote{This can be seen by using a weak or distributional meaning of the Laplacian 
$$
0 \geq  \int \nabla u \cdot \nabla \varphi  = 
\int_{\{u  > u(x^0)\} \cup \{u  < u(x^0)\} } \nabla u \cdot \nabla \varphi =
- \int_{ \{ u = u(x^0)\}} \partial^+_\nu u \varphi + \int_{\{ u = u(x^0)\}} \partial^-_\nu u \varphi 
$$ 
for $\varphi \in W_0^{1,2} (B_r(x^0) )$, and where $\partial^\pm_\nu$ denotes the normal derivatives on $\{ u= u(x^0)\}$ from $\{\pm u > u(x^0)\}$, and  $\nu = \nabla u /|\nabla u|$ on $\{ u = u(x^0)\}$.  Since $\varphi$ is arbitrary, we arrive (formally)  at \eqref{eq:Lap-pos}.
}
\begin{equation}\label{eq:Lap-pos}
    0 \leq \Delta u (x^0) = \partial^+_\nu u(x^0) - \partial^-_\nu u(x^0),
\end{equation}
where $\pm$ denotes the derivative taken from the regions ${\pm (u(x) - u(x^0)) > 0}$. This shows that the normal derivative seen from the super-level sets is greater than that of  the sublevel sets. In particular, if $u$ is nonnegative and vanishes outside a region $D$, then using the lower semi-continuous envelope fails to capture the correct behavior of $u$ relative to its support.

The justification against  the second objection is that we aim to capture the behavior of the gradient at least in a non-tangential sense. For instance, in the scenario just described with $D = \{u > 0\}$ and $u \geq 0$ elsewhere, the boundary of $D$ may be smooth at $x^0$ (e.g., touched by spheres from both sides) but still exhibit a zig-zag geometry. At such inward corners of the boundary, the gradient may blow up, so that the upper semi-continuity of the gradient fails if the boundary gradient is defined through non-tangential limits.
When $D$ is convex, the situation is more transparent, and one expects the gradient to behave better at the boundary. In fact, in \cite{LW2006}, the authors  construct an example where the gradient, taken non-tangentially, is nonzero at $x^0$ but is not continuous, since  tangentially, along outward corners, the gradient vanishes. Nevertheless, they prove that the solution to their PDE is differentiable. 

Further results in the literature concern the differentiability of  the boundary gradient of solutions to PDEs, under one-sided smoothness assumptions.\footnote{It is obvious that differentiability, unless uniform in a uniform neighborhood, does not give continuity of the gradient.}
We refer to several recent works in this direction\cite{HLW2014, HLZ2016, HLW2020}.

Before discussing our main results, we also want to draw the reader's attention to an example, provided by 
H.W. Alt and L.A. Caffarelli \cite[Example 2.7]{Alt-Caff}, in a complete different context,  concerning a  Lipschitz harmonic function  in a cone, with the boundary at the  origin being non-smooth and such that $|\nabla u|$ has different values, depending on which direction one approaches the origin. Their example is the following function 
\[
u(x) := r \max\left( \frac{f(\theta)}{f'(\theta_0)}, \; 0 \right),
\]
where $\theta$ is a  given  polar  angle  and 
\[
x = r(\cos \phi \sin \theta, \; \sin \phi \sin \theta, \; \cos \theta), \qquad 
f(\theta) = 2 + \cos \theta \, \log \left( \frac{1 - \cos \theta}{1 + \cos \theta} \right)
\]
is a solution of
\[
(\sin \theta \, f')' + 2 \sin \theta \, f = 0, 
\quad f'\!\left(\tfrac{\pi}{2}\right) = 0,
\]
and where $\theta_0 \approx 33.534^\circ$ is the unique zero of $f$ between $0$ and $\tfrac{\pi}{2}$. 
The function $u$ is harmonic in $\{u>0\}$ and $|\nabla u | = 1$ on $\partial\{u>0\}\setminus\{0\}$. 
The support of $u$ has a singularity at the origin, and 
\[
|\nabla u| \leq 1 - \epsilon 
\quad \text{for} \quad 
\theta_0 + \epsilon \leq \theta \leq \pi - \theta_0 - \epsilon,
\]

This example and the above discussions  indicate that one needs strong enough assumptions from one side of the level sets to obtain any reasonable result in our problem. One such condition seems to be the exterior touch with a reasonably smooth graph at a point to obtain further improved properties for the gradient. Indeed, this assumption allows us to use a barrier argument and obtain Lipschitz regularity of our subharmonic function in \eqref{eq:sub}. See Lemma \ref{lem:barrier} in the Appendix.

Since Lipschitz functions are almost everywhere smooth, 
our next task now  is to provide a definition of the gradient at   points where $u$ is not $C^1$, and in such a way that it coincides with the gradient at its  $C^1$ points.
Indeed, consider the  weighted  integral  mean of the gradient on the super-level sets of $u$, defined as 
\begin{equation}\label{eq:monot0}
I(r, y, u) := \frac{1}{r^2} \int_{B_r(y)} \frac{|\nabla (u(x) -u(y))_+|^2}{|x - y|^{n-2}} \ dx,
\end{equation}
where $(u(x) -u(y))_+ = \max (u(x) -u(y), 0)$. We will  prove that under an extra condition on the level sets of $u$, this function is  non-decreasing in $r$, and therefore it has a limit as $r$ tends to zero. For $C^1$ points this converges to $c_y |\nabla u(y)|^2$, for some $c_y >0$. We will also show  that $c_y$ is independent of $y$, and hence it is constant.
From this we shall then deduce the upper semi-continuity of the gradient.

Our argument is based on what is known in the literature as the  ACF monotonicity formula \cite{CJK}, which is well-developed for more general scenarios; 
see \eqref{eq:property}--\eqref{eq:monot1}.

In particular, this allows us to generalize our results for other operators (including time dependent operators) where such a monotonicity formula is available. 
As we also mentioned earlier, the approach here  is applicable to the Dirichlet problem, when the boundary of the domain possesses the aforementioned property.

\section{Main Results}

For our bounded subharmonic function $u$ in \eqref{eq:sub},
suppose that  for every point $y \in B_{2/3}$, the super-level set
\begin{equation}\label{eq:Dy}
    D_{y}^+ := \{ x \in B_1 : u(x) > u(y) \}
\end{equation}
can be touched from the exterior  at $y$ by a connected  domain  $K_y \subset B_2$ with $K_y \setminus \overline B_{3/2} \neq \emptyset $ and   for some  $r_0 > 0$ (fixed)
with  
\begin{equation}\label{eq:norm}
    \partial K_y \cap B_{r_0} (y) \in C^{1,\text{Dini}}(B_{r_0}(y)) ,
    \end{equation}
with  the norm being uniform in $y$.

  By  a barrier argument, this geometric condition implies that $u$ is uniformly Lipschitz; see  Lemma \ref{lem:barrier} in the Appendix.

To invoke our main analytical tool, the ACF-monotonicity formula, for a fixed $y\in B_{2/3}$ we consider  a pair of non-negative continuous  functions $h_\pm$ satisfying 
\begin{equation}\label{eq:property}
h_+ h_- = 0, \qquad h_+(y) = h_-(y) = 0, \qquad \Delta h_\pm \geq 0 \quad \text{in } B_{r_1 } (y) , \qquad r_1 < 1 -|y|.
\end{equation}

We now define
\begin{equation}\label{eq:monot2}
I(r, y, h_\pm) := \frac{1}{r^2} \int_{B_r(y)} \frac{|\nabla h_\pm(x)|^2}{|x - y|^{n-2}} \ dx,
\end{equation}
where the  integrals are meaningful and well defined, by  \cite[Remark 1.5]{CJK}.
Next we define the monotone increasing function (in $r$)
\begin{equation}\label{eq:monot1}
\mathcal{I}(r, y, h_+, h_-) := I(r, y, h_+) I(r, y, h_-) \quad \nearrow \quad \hbox{in } r,
\end{equation}
where the monotonicity follows from  \cite{CJK}.

With this monotonicity formula at hand and the imposed geometric property on  super-level sets, 
we can define the gradient of our function $u$ from \eqref{eq:sub} at every point in $B_{1/2}$, in a unique way.
 Indeed, let's fix $y\in B_{1/2}$ and consider, by the above geometric assumption, a domain 
  $K_y $ with $C^{1,\text{Dini}}$ boundary that touches $\partial D_y^+$ at $y$.

Let now $G_y$ denote the harmonic function in $K_y$ with zero boundary data on $\partial K_y \cap B_{1}$, with $G_y = 1$ on $\partial K_y \setminus B_{5/4}$, and defined arbitrarily on the rest of the boundary in such a way that the boundary value is smooth, and non-negative.

Obviously, defined in this way,  $G_{y}$ is uniformly $C^1$ in $\overline{K_y} \cap B_{r_0/2} (y)$, so its gradient is well defined up to the boundary $\partial K_y \cap B_{r_0/2} (y)$. Therefore, the limit  
\[
\lim_{r \to 0} I(r, y, G_{y}) = c_0|\nabla G_y(y)| >0 
\]  
is well defined, and strictly positive by Hopf's boundary point Lemma, see \cite{Widman}, and \cite{Ap-Naz} for divergence type operators, which will be needed later for our results in a general setting.
Here  $c_0$ is defined as 
\begin{equation}\label{eq:c0}
c_0 :=
 \frac12 \int_{B_1 (0)   } \frac{dx }{|x |^{n-2}} .
\end{equation}
 
Next, we define $u_y := \max(u - u(y), 0)$, which put in the   ACF-monotonicity formula along with $G_y$, 
gives us 
\begin{equation}\label{eq:I}
\lim_{r \to 0} I(r, y, u_y) = \lim_{r \to 0} \frac{\mathcal{I}(r, y, u_y, G_y )}{I(r, y, G_y )},
\end{equation}
which also exists.
Moreover, this limit is independent of the particular choice of $K_y$, or the harmonic function $G_y$.
Thus, the limit is well defined.\footnote{Indeed, one can verify that $I(r, y, u) = \frac{\mathcal{I}(r, y, u_y, h)}{I(r, y, h)}$ for any admissible choice of $h$, as long as the integrals are well defined.}

Now,  we define $|\nabla u (y)| $  
as 
\begin{equation}\label{eq:gradient}
    |\nabla u(y)|^2 := \frac{1}{c_0} \lim_{r \to 0} I(r, y, u_y), \qquad \forall\, y \in   B_{1/2},
\end{equation}
where $c_0$ is defined in \eqref{eq:c0}. We shall use this  to also define $\nabla u (y) $ below.

It should be noted that by this definition, 
at points where $\nabla u $ is  continuous, we recover the gradient in its pointwise meaning.

From the above discussion, we conclude  that  for any $0 < r_1 < 1/2 $  and any $y \in  B_{1/2}$ 
$$c_0^2|\nabla u (y) |^2|\nabla G_y (y) |^2 = 
\lim_{r\to 0}  \mathcal I(r,y,u_y,G_y) \leq   \mathcal I(r_1,y,u_y,G_y) .
$$
Fix a small parameter $0 < \epsilon \ll 1$, and let $y^0 \in B_{1/2}$ be arbitrary.   Consider any point $y$ such that 
\begin{equation}\label{eq:distance}
 |y - y^0| < \epsilon^2, \qquad 
\max_{z\in \partial D_{y}^+}{\hbox{dist}(z, \partial D_{y^0}^+)} < \epsilon^2
\end{equation}
which,  eventually, is to  approach $y^0$. Our goal is to prove
 the upper semi-continuity of $|\nabla u|$ at $y^0$, i.e., 
 $$\limsup_{y \to y^0} |\nabla u(y)| \leq  |\nabla u(y^0)|.$$
 
  Obviously,  we may only consider those $y$ in the vicinity of $y^0$ in the sense of \eqref{eq:distance}
and  such that  
$$      |\nabla u(y)| >  0 , $$
 which, along with    the monotonicity of ${\mathcal I} (r, y, u_y, G_y)$
and that $|\nabla G_y(y)| > 0$, implies
\begin{equation}\label{eq:estimate0}
    c_0^2|\nabla u (y) |^2|\nabla G_y (y) |^2
\leq   \mathcal I(r_\epsilon,y,u_y,G_y) = I(r_\epsilon,y,u_y) I(r_\epsilon,y,G_y)  ,
\end{equation}
where we take $r_\epsilon = \epsilon + |y - y^0| < 2\epsilon $.

We shall now compute  the right hand side in the above. 
Indeed, it is straightforward  (see Appendix, Lemma \ref{lem:estimates}) that 
\begin{equation}\label{eq:estimate}
    I(r_\epsilon,y,u_y) =  I(r_\epsilon,y^0,u_y) + o_\epsilon (1).
\end{equation}
This in particular implies
\begin{equation}\label{eq:estimate2}
c_0^2|\nabla u (y) |^2|\nabla G_y (y) |^2 
\leq  I(r_\epsilon,y^0,u_y) I(r_\epsilon,y,G_y) +  o_\epsilon (1) .
\end{equation}
If \( |\nabla u(y^0)| = 0 \), (as the integral mean defined in \eqref{eq:gradient})
then the first term on the  right-hand side of \eqref{eq:estimate2} is small, and we obtain the estimate 
\[
c_0^2 |\nabla u(y)|^2 |\nabla G_y(y)|^2 = o_\epsilon(1).
\]
Now, due to the geometric assumption that $K_y$ has a uniform $C^{1,Dini}$ norm in a uniform neighborhood, independent of $y$, see \eqref{eq:norm}, 
we must have 
\( |\nabla G_y(y)| > c_1 > 0 \), and  it follows that
$|\nabla u(y)|^2 = o_\epsilon(1)$,
which establishes the upper semi-continuity of \( |\nabla u| \) at \( y^0 \), for the case \( |\nabla u(y^0)| = 0 \).

Next, consider the case \( |\nabla u(y^0)| > 0 \).
Using the definition of $u_y$ along with  \eqref{eq:distance} and  $r_\epsilon = \epsilon + |y-y^0| \leq  \epsilon + \epsilon^2$, we see that  
if $u(y) \leq u(y^0)$ then 
$$
I(r_\epsilon,y^0,u_{y}) = 
 I(r_\epsilon,y^0,u_{y^0}) + \frac{1}{r_\epsilon^2 }\int_{B_{r_\epsilon} (y^0) \cap D_{y}^+  \setminus D_{y^0}^+ }  \frac{|\nabla u|^2}{|x-y^0|^{n-2}} =  I(r_\epsilon,y^0,u_{y^0}) + o_y (1) .
$$
On the other hand if if $u(y) \geq u(y^0)$ then 
$$
I(r_\epsilon,y^0,u_{y}) = 
 I(r_\epsilon,y^0,u_{y^0}) - \frac{1}{r_\epsilon^2 }\int_{B_{r_\epsilon} (y^0) \cap D_{y^0}^+  \setminus D_{y}^+ }  \frac{|\nabla u|^2}{|x-y^0|^{n-2}} =  I(r_\epsilon,y^0,u_{y^0}) + o_y (1) .
$$

 Adding up, we have 
 $$
c_0^2|\nabla u (y) |^2|\nabla G_y (y) |^2 
\leq I(r_\epsilon,y^0,u_{y^0})  I(r_\epsilon,y,G_y) + o_y(1)o_\epsilon (1) ,
$$
and therefore,
$$c_0^2|\nabla u (y) |^2 \leq  \frac{I(r_\epsilon, y^0, u_{y^0}) I(r_\epsilon, y , G_{y} ) + o_y(1)o_\epsilon (1)  }{|\nabla G_y(y)|^2}
=   \frac{I(r_\epsilon, y^0, u_{y^0}) \left(c_0|\nabla G_{y} (y) |^2\right) + 
o_\epsilon (1) (1+   o_y(1))  }{|\nabla G(y)|^2},
$$
where we have used that  $G_y$, being  uniformly $C^1$,  can be expressed as 
$$
I(r_\epsilon, y , G_{y} ) = c_0|G_y(y)|^2 + o_{\epsilon}(1).
$$
We thus have 
$$c_0^2|\nabla u (y) |^2 \leq 
c_0I(r_\epsilon, y^0, u_{y^0})  + 
\frac{o_\epsilon (1) (1+   o_y(1))  }{|\nabla G_y(y)|^2}.
$$

Letting $\epsilon$, and hence $|y-y^0|$, tend to zero, we have 
$$
\lim_{y \to y^0}|\nabla u (y) |^2 \leq 
\frac{1}{c_0} I(0_+, y^0, u_{y^0})=|\nabla u (y^0)|^2.
$$

Next, we  prove the  upper semi-continuity for each directional derivative $\partial_ j u$ for $j=1, \cdots , n$.
We assume, after rigid motion, that   the outward normal direction to  $\partial K_{y^0}$ is $e_1$; this means that $e_1$ is pointing inward $D_y^+$.
By a scaling and blowup argument one can  prove that either $|\nabla u (y^0)| =0$, or   
\begin{equation}\label{eq:asymptotic}
    u(x) = c_1(x_1 -y_1^0)^+ + o(1) \qquad \hbox{in } D^+_{y^0} , \quad \hbox{near } y^0.
    \end{equation}
 This, in turn,  implies  $|\nabla u (y^0)|= \partial_1 u(y_0) $  where the latter is also defined in the sense of limit of  integral-mean. 
See   Lemma \ref{lem:asymptot} in  Appendix.

  Now for any point $y \in D$ with $|\nabla u (y) |>0$, we have $|\partial_1  u(y) |\leq |\nabla u (y) |$, and hence we can run the above  upper semi-continuity argument for $|\nabla u (y) |$, and obtain 
$$
\lim_ {y \to y^0}|\partial_1  u(y) |^2\leq 
\lim_ {y \to y^0}|\nabla u (y) |^2 \leq |\nabla u (y^0) |^2.
$$
In case $|\nabla u (y^0) | = 0$, we are done. 
Otherwise  $|\nabla u (y^0)| = \partial_1 u (y^0)$, and  (in the integral-mean sense) $\partial_j u (y^0)= 0$  when $j\neq 1$.

The argument, with some adjustment,  can be applied to subsolutions of the  Poisson equation  
\begin{equation}\label{eq:sub2}
    \Delta u \geq -1 ,
\end{equation} 
where one needs to replace the ACF monotonicity formula \eqref{eq:monot1}
with that of  Caffarelli-Jerison-Kenig, which says
\begin{equation}\label{eq:alm-mon}
    \mathcal I (\rho ,y,h_+,h_-) \leq (1 + r)\mathcal I (r ,y,h_+,h_-) + C r^\delta ,
\end{equation}
for $0 < \rho \leq r < 1 $, some $C >0$
and $\delta >0 $, such that $u \in C^{0,\delta} (B_1)$. In our case, we have  $\delta =1$, since $u$ is Lipschitz, due to assumptions. 

Obviously \eqref{eq:alm-mon} implies that $\mathcal I (r ,y,h_+,h_-) $ converges  as $r\searrow 0$. It is also apparent that the whole argument run so far only uses the fact that this limit exists. We can thus run the argument again with this property and obtain the upper semi-continuity for the gradients of solutions to \eqref{eq:sub2}, under the extra geometric assumption.

We have thus proved the following theorem.

\begin{thm}\label{thm:subsol1}
    Let $u$ be a bounded  function satisfying  $\Delta u \geq -1$ in $B_1$,  and suppose  the exterior touching property \eqref{eq:norm} holds for all points in $B_{3/4}$.
Then $\nabla u$ is upper semi-continuous in $B_{1/2}$.
\end{thm}

The above theorem applies to the case of the heat operator, when the time derivative is bounded from below, since in that case we fall within the assumptions of the theorem, with $\Delta u \geq -\partial_t u -1 \geq -C$. It can further be generalized to divergence type operators with lower order terms.
We formulate this as a corollary.

For this purpose, we define $Q_r$ as the unit cylinder  $Q_r= B_r \times (0,1)$ and the  parabolic operator
\begin{equation}\label{eq:F}
    F[u] := \inf_{\alpha \in I} F_{\alpha} [u],
\end{equation}
where $I$ is an index set, and  $F_\alpha [u]$ are divergence type operators
$$
F_\alpha [u]  =div(A_\alpha \nabla u) + b_\alpha \cdot \nabla u + c_\alpha  u - d_\alpha \partial_t u ,
$$
where  $d_\alpha=d_\alpha(x,t)  > 0$, $A_\alpha=A_\alpha(x,t)$, $b_\alpha=b_\alpha(x,t)$,  $c_\alpha=c_\alpha(x,t)$  are bounded, 
while  $A_\alpha $  satisfies uniform ellipticity and is continuous with double-Dini modulus of continuity;\footnote{A modulus of continuity $\omega (r)$ is  Double-Dini if 
$\int_0^1 \frac{1}{r}\left( \int_0^r \frac{\omega(t)}{t}\,dt \right) dr < \infty.$
 }
see  assumptions in \cite{Mat-Pet}. Here, and later, for the time dependent problems, we shall  still denote the space gradient by $\nabla $. 

Then the following corollary holds. 

\begin{cor}\label{cor:heat}
In $Q_1$,  let $u$ satisfy         $F[u]  \geq -1, $ (see \eqref{eq:F} 
)   and suppose that for every point $(y,s) \in Q_{3/4}$, the  set $\{ u(x,t) > u(y,s)\}  \cap \{t=s\} $
can be touched from the exterior at $y$ by a (spatial) domain  $K_y$ with $\partial K_y \cap Q_{1/2}$ being 
uniformly  $C^{1,\text{Dini}}$ in $x$-variable. Suppose further that $\partial_t u \geq -C_0$.
Then $\nabla u (x , s)$ is upper semi-continuous in $Q_{1/2}$, for each fixed $s$.
\end{cor}

\begin{proof}
By definition, we see that for any $\alpha \in I$, we have 
$F_\alpha[u] \geq  F [u]$, and since  the time derivative is bounded from below, we have for a fixed  $\alpha_0 \in I$
$$
div(A_{\alpha_0}  \nabla u)  + b_{\alpha_0} \cdot \nabla u \geq  - c_{\alpha_0}  u +  d_{\alpha_0} \partial_t u -1 \geq  - c_{\alpha_0}  u -d_{\alpha_0} C_0 -1
$$
and since $u$ and $c_{\alpha_0} $ are bounded, we have 
$$
div(A_{\alpha_0}  \nabla u)  + b_{\alpha_0}  \cdot \nabla u \geq   -C_1.
$$
As in  the Laplacian case, we need  a barrier argument to conclude that $u$ is Lipschitz in $x$. 
For the divergence operator,  
$div(A_{\alpha_0} \nabla u)   + b_{\alpha_0}\cdot \nabla u $, the existence of weak solutions to the (elliptic)  Dirichlet problem, serving as a barrier to $u(x,s)$, with $s$ fixed, follows from the standard minimization or Lax-Milgram theory; see e.g. \cite[Section 6.2]{Evans}.
The up to the boundary  $C^1$-smoothness of solutions  for $C^{1,Dini}$ boundaries also follows from    \cite[Theorem 1.3]{Dong-2018}.\footnote{A parabolic version of this theorem also exists, see  \cite[Theorem 1.1]{Dong-2021}.} 
In particular, we have 
$$
div(A_{\alpha_0}  \nabla u ) \geq -     b_{\alpha_0} \nabla u   -C_1 \geq -C_2
$$

Since we assume $A_{\alpha_0}$ to be double-Dini, we can carry out   the almost monotonicity formula of   \cite[Theorem IV]{Mat-Pet} in the same manner as we did earlier. 
\end{proof}

\begin{rem}
    The above arguments will work also for general operators $F_\alpha [u]$, admitting an ACF-type (almost) monotonicity formula.  
\end{rem}

\section{Application to the Dirichlet problem}

Using ideas from our analysis, we formulate the following theorem for solutions to the Dirichlet problem. To state the result, we let  $Q_r = B_r \times (-1,1)$, and $\Omega $
be a domain in $ \R^{n+1}$ and define the operator 
$$
{\mathcal L}u= div(A \nabla u) + b \cdot \nabla u + c  u -  \partial_t u ,
$$
where $A,b,c$ are bounded and $A$ satisfy uniform ellipticity, and is continuous with double-Dini modulus of continuity.
Moreover, let  $u$ be a solution of
\begin{equation}\label{eq:pde}
    {\mathcal L}u = f  \quad \text{in } Q_1 \cap \Omega,
\qquad (0,0) \in \partial \Omega
\end{equation}
with Dirichlet boundary data $g$, only on $\partial \Omega \times (-1,1)$; this allows us to also apply our next theorem to the elliptic case. In what follows below, we let $C$ be a universal constant depending on the involved ingredients and may change value from one appearance to next.
We assume the following:

\begin{itemize}
\item[(i)] The operator ${\mathcal L}$ and the function $f$ are  such that the solution $u$ to \eqref{eq:pde} is $C^{1,Dini} (\Omega  \cap Q_{9/10})$ with respect to  spatial variables.

\item[(ii)] The time derivative satisfies $\partial_t u \geq -C$.

\item[(iii)] $g \in C^{1,Dini}_x$, $div(A \nabla g) \leq C$, and $\partial_t g \geq -C$. 

\item[(iv)]  $f \geq -C  $.

\item[(v)] The principal  part $div(A \nabla u)$ admits an elliptic ACF (almost) monotonicity formula, as in \cite[Theorem IV]{Mat-Pet}.

\item[(vi)] For every $y \in \partial \Omega \cap B_{4/5}$, the (spatial) gradient $\nabla u(y,s)$ is defined as above using the ACF-monotonicity formula; see \eqref{eq:gradient} and the discussion preceding \eqref{eq:sub2}.
\end{itemize}

Under these assumptions, we have the following theorem.

\begin{thm}\label{thm:bdry}
Let $u$ be a solution to  the  Dirichlet problem in $\Omega$ as stated above, with ${\mathcal L}$ having the properties mentioned there.
Assume that there is $r_0 > 0$, such that 
for every boundary point $(y,s) \in \partial \Omega \cap Q_{3/4}$, 
the domain $\Omega$ can be touched from the exterior at $y$ by a spatial domain 
$K_y$ whose boundary 
$\partial K_y \cap Q_{r_0} (y)$ is uniformly $C^{1,\mathrm{Dini}}$ in the $x$-variable. 
Then  for every fixed $t$, $\nabla u (x,t)$ is upper semi-continuous in the $x$-variable in 
$\overline\Omega \cap Q_{1/2}$.
\end{thm}

The requirement that  the boundary be $C^{1,Dini}$ from the exterior at all points close to $y$, is strict, since otherwise it may be seen that the exterior of the zig-zag domain (mentioned in the introduction) can serve as  the domain $\Omega$, and the theorem fails  for this domain.

\begin{proof}
By standard elliptic/parabolic theory $u$, and hence $u-g$, are $C^{1,Dini}$ in $x$-variable in  $\Omega \cap Q_ {9/10}$, and uniformly Lipschitz up to the boundary. Define   $v:=u-g$, and use  properties assumed for $g$ to arrive at
$${\mathcal L}v = {\mathcal L}(u-g) = f  -\hbox{div} (A\nabla g  ) -b \cdot \nabla g - cg + \partial_t g 
\geq  - C ,
$$
with $v$ being $C^1$ in $x$-variable in  $\Omega \cap Q_ {9/10}$.
 Similarly as earlier, using that solution is uniformly Lipschitz  in $x$-variable in  $\overline \Omega$, and that $\partial_t u \geq -C$,  
 we can conclude that in $\Omega \cap Q_{99/100}$
\begin{equation}\label{eq:-C}
    div(A \nabla v)   \geq    -C.
\end{equation}

Let us now fix a time level, say $t=s$, and any  boundary point $(y^0,s) \in \partial \Omega \cap Q_{1/2}$. 
Using ACF-monotonicity formula, we define $\nabla v $ at the  boundary points in a similar way to what  we did in \eqref{eq:gradient}.

We now prove the upper semi-continuity at $(y^0,s)$
for points $(y,s) \in  \partial \Omega$, with $y$ approaching $y^0$. To do this we may, using the exterior touching property at the boundary, assume that at each boundary point we have a touching domain $K_y$ which is a translation and rotation of $\{ x: \ x_1 > |x'|\omega (|x'|)\}$, with $\omega (r)$ being a Dini modulus of continuity, and independent of the point $(y,s)$. 

For simplicity of notation, we suppress the $s$-variable, and 
write $V(x) = v(x,s)$.

Now, the time $s$ being fixed, we apply the monotonicity formula to  $V$ and $G_y$, as before, at $y$, assuming $|y - y^0|< \epsilon^2$,    and $r_\epsilon= \epsilon + \epsilon^2$ and small. Observe that $V =0 $ in $\Omega^c \cap \{t=s\}$ so there is no need to denote it by $V_y$, as we did in the previous section. Indeed, as in \eqref{eq:estimate0} we have 
\begin{equation*}
    c_0^2|\nabla V (y) |^2|\nabla G_y (y) |^2
\leq   \mathcal I(r_\epsilon,y,V,G_y) = I(r_\epsilon,y,V) I(r_\epsilon,y,G_y)  .
\end{equation*}
We observe that $I(r_\epsilon,y,G_y) = I(r_\epsilon,y^0,G_{y^0})$
since  
 $K_y$ is a rotation and translation of $K_{y^0}$ and hence the integral  $I(r_\epsilon,y,G_y)$ does not change under this transformation. Therefore 
\begin{equation}\label{eq:estimate3}
    c_0^2|\nabla V (y) |^2|\nabla G_y (y) |^2
\leq   \mathcal I(r_\epsilon,y,V,G_y) = I(r_\epsilon,y,V) I(r_\epsilon,y^0,G_{y^0})  .
\end{equation}
The rest of the proof follows in the same way as in the previous section. 
In particular, we have the  upper semi-continuity of $|\nabla v(y,s)|$ along the boundary, for fixed $s$.  

Now, as in the previous section, using blow-up and the argument in Lemma \ref{lem:asymptot} in Appendix we conclude that $\partial_j v$ are upper semi-continuous, for all directions $j =1, \cdots ,n$.

We   consider points $(y,s) \in \Omega$, and close to $(y^0,s)$, and denote 
$${\bf a} := \nabla v (y^0, s).$$
It suffices to show that  for $\delta >0$ there is $\epsilon >0$ such that for 
$|y-y^0 |< \epsilon $ we must have 
$|\nabla V (y) -{\bf a}| < \delta $.
Suppose, on the contrary, that this fails. Then there is a sequence $y^j \to y$, such that $|y^j - y^0 |< 1/j $, but 
\begin{equation}\label{eq:bigger}
    |\nabla V (y^j) -{\bf a}| \geq  \delta .
\end{equation}

Define the rescaled functions
\[
V_j(x) = \frac{V(r_j x + y^j)}{r_j},
\qquad
r_j = \operatorname{dist}(y^j, \partial D),
\]
and let   $z^j \in \partial \Omega$ be  any of the  closest boundary point to $y^j$ (to be used below).
Since $V$ is uniformly Lipschitz, the sequence $\{V_j\}$ is uniformly bounded and 
uniformly Lipschitz in the balls $B_{1/(2r_j)}$.

We now consider the geometry of the domain. Because $\Omega$ admits exterior touching by $C^{1,\mathrm{Dini}}$ surfaces, the rescaled domains
\[
\Omega_j = \frac{1}{r_j}\bigl(\Omega - z^j\bigr),
\]
 converge (up to a subsequence) to a limit domain $\Omega_0$. Correspondingly, the functions $v_j$ converge locally uniformly to a limit function $v_0$.

The exterior regularity of $\Omega$ implies that, after a suitable rotation and translation, the limit domain satisfies
\begin{equation}\label{eq:omega0}
    \Omega_0 \subset \{x_1 > 0\}.
\end{equation}
Moreover, $B_1 \subset \Omega_j$,  the points 
$y^j$ translate to ${\bf e}_1= (1,0, \cdots , 0)$,  $z^j $ translate  to the origin, 
and 
\[\lvert \nabla V_j({\bf e}_1) - {\bf a} \rvert = 
\lvert \nabla V(y^j) - {\bf a} \rvert \geq \delta.
\]
On the other hand,  since $y^j \to y^0 $, we also have $z^j \to y^0$, and hence for large values of $j$ we must have 
\begin{equation*} 
  |\nabla V_j (0) - {\bf a} | =   |\nabla V (z^j) -{\bf a}| <  \delta/4 ,
\end{equation*}
As $j $ tends to infinity, 
these two inequalities turn into 
\begin{equation} \label{eq:inequalies}
  \lvert \nabla V_0({\bf e}_1) - {\bf a} \rvert \geq  \delta  > \delta/4 > 
  |\nabla V_0 (0) - {\bf a} | .
\end{equation}

Now we study the limit function $V_0$, in the light of these inequalities. By \eqref{eq:omega0} we have already established that $\Omega_0 \subset \{x_1 > 0\} $. It is also apparent that $V_0$ satisfies a limit PDE, with constant coefficient, due to the scaling and continuity of the coefficients. In particular we have 
$$
div(A_0 \nabla V_0)   \geq    0,
$$
where $A_0$ is a $A(0)$.
Observe that the linear scaling (and blow-up)  would annihilate the constant $-C$ to the right hand side of the inequality in \eqref{eq:-C}.

We now reflect $v_0$ in the plane $\{x_1 =0\}$ and  denote it by $\tilde V_0$, which along with $V_0$ in the monotonicity formula gives us that either $V_0$ is a linear function or identically zero. Either of these cases are  in conflict with \eqref{eq:inequalies}.

\end{proof}

\bigskip

\appendix

\section{}

Here, we give proofs of those technical lemmas that are needed in the paper. Although the proofs are elementary,  we include them  for the reader's convenience. 

\begin{lem}\label{lem:barrier}(Barrier argument)
Let $u$ satisfy 
$$div(A \nabla u) \geq -C, \qquad \hbox{in } B_1,
$$
where $A$ satisfies a uniform ellipticity condition. Then, under the geometric assumption introduced in this paper, see \eqref{eq:Dy}--\eqref{eq:norm}, the function $u$ is Lipschitz.
\end{lem}

\begin{proof} We adapt the notation used earlier, and consider a point $y^0 \in B_{1/2}$, and the super-level set $D^+_{y^0}$.
     Since for each $y^0 \in B_{r_0}$ we have a touching $C^{1, Dini}$ domain $K_{y^0}$ in $B_{r_0}(y^0)\setminus D^+_{y^0}$, for some $r_0$ independent of $y^0$, 
     we may consider a  function $h$ in $K_{y^0}^c \cap B_{1/2} (y^0)$, with boundary values $u$ on $\partial B_{r_0} (y^0) \cap K^c_{y^0}$,  $h= u(y^0)$ on $\partial K^c_{y^0} \cap  B_{r_0} (y^0)$, and $div(A \nabla h) = -C$ in $K_{y^0}^c \cap B_{r_0} (y^0)$.
By the comparison principle $u \leq h$. Since $h$ has linear growth from the boundary, it implies that $u$ has sub-linear growth, from $y^0$. This can be done for all boundary points on the super-level sets, implying that $u$ has sub-linear growth from the boundary, while looking from super-level sets. In conclusion we have
\begin{equation}\label{eq:Lip-est}
    u(y) \leq u(y^0) + C|y^1-y^0|, \qquad \forall \ y^1 \in   B_{r_0}(y^0) \cap D^+_{y^0}.
\end{equation}
Now for points $y^1$ in $B_{r_0}(y^0) \setminus  D^+_{y^0}$, we consider the super-level set 
$D^+_{y^1}$ and argue in the same way to arrive at 
$$
u(y^0) \leq u(y^1) + C|y^1-y^0|,  \qquad \forall \  y^1 \in B_{r_0}(y^0) \setminus  D^+_{y^0}
$$
which together with the estimate \eqref{eq:Lip-est} implies $u$  has a  linear growth from $y^0$. Observe that the Lipschitz constant $C$ is independent of $y^0$.
This applies to all points in $B_{1/2}$, and 
therefore $u$ is uniformly  Lipschitz in $B_{1/2}$.

\end{proof}

\begin{lem}\label{lem:estimates}
 We prove    \eqref{eq:estimate}, i.e.
 \begin{equation}
    I(r_\epsilon,y,u_y) =  I(r_\epsilon,y^0,u_y) + o_\epsilon (1).
\end{equation}
\end{lem}

\begin{proof}
This can be done by a direct calculation. We have 
$$
\Big| I(r_\epsilon,y,u_{y}) -  I(r_\epsilon,y^0,u_y)| \le
\frac{1}{r_\epsilon^2 }\int_{B_{r_\epsilon} (y) \setminus 
B_{r_\epsilon} (y^0)}  \frac{|\nabla u|^2}{|x-y|^{n-2}} + 
\frac{1}{r_\epsilon^2 }\int_{B_{r_\epsilon} (y^0) \setminus 
B_{r_\epsilon} (y)}  \frac{|\nabla u|^2}{|x-y^0|^{n-2}}
+
$$
$$
+\frac{1}{r_\epsilon^2 }\int_{B_{r_\epsilon} (y^0) \cap 
B_{r_\epsilon} (y)} |\nabla u|^2\left|\frac{1}{|x-y|^{n-2}}- \frac{1}{|x-y^0|^{n-2}}\right|.
$$
Since $|y -y^0 | < \epsilon^2$,  the  integration domains in 
the first two integrals on the right side of the inequality are small, and tend to zero as $\epsilon $ tends to zero.
 As to the last integral, it can be estimated in the following way, using the mean-value theorem and \eqref{eq:distance}:
$$
\frac{1}{r_\epsilon^2 }\int\limits_{B_{r_\epsilon} (y^0) \cap 
B_{r_\epsilon} (y)} |\nabla u|^2\left|\frac{1}{|x-y|^{n-2}}- \frac{1}{|x-y^0|^{n-2}}\right| \le C  \frac{|y-y^0|}{r_\epsilon^2 }\int\limits_{B_{r_\epsilon} (y^0) \cap 
B_{r_\epsilon} (y)} \frac{1}{\min\{|x-y|, |x-y^0|\}^{n-1}}\le 
$$
$$
\le 2C
\frac{|y-y^0|}{r_\epsilon^2 }\int_{B_{r_\epsilon} (y^0)} \frac{1}{|x-y^0|^{n-1}} =  C_1\frac{|y-y^0|}{r_\epsilon} = 
C_1\frac{|y-y^0|}{\epsilon+|y-y^0|}\le C_1\frac{|y-y^0|}{\epsilon}\le C_1\frac{\epsilon^2}{\epsilon} = C_1\epsilon. 
$$

\end{proof}

\begin{lem}\label{lem:asymptot}(Asymptotic development)
Under the assumptions of Theorem \ref{thm:subsol1}, or Corollary \ref{cor:heat}, if 
 $|\nabla u (y^0)|\neq 0$,  
  then, after a rigid motion, 
\begin{equation}\label{eq:asymptotic2}
    u(x) = c_1(x_1 -y_1^0)^+ + o(1) \qquad \hbox{in } D^+_{y^0} , \quad \hbox{near } y^0.
    \end{equation} 
\end{lem}

\begin{proof}
    
By the monotonicity formula we have, for $r,s > 0$
$$
\mathcal{I} (rs,y^0, u_y, G ) = \mathcal{I} (s,0, (u_r)_y, G_r ) ,
$$
where $G$ is a solution to the corresponding PDE, already defined earlier in the paper, and $u_r (x) = u(rx + y^0)/r$, and $G_r (x) = G(rx + y^0)/r$. 
Both $u, G$ being Lipschitz, are in $W_{loc}^{1,\infty} (B_{1/2}) $, and hence $u_r, G_r$ are in $W_{loc}^{1,\infty} (B_{1/2r}) $. In particular, for a subsequence $u_{r_j}, G_{r_j}$ converge weakly in $W_{loc}^{1,p} (\R^n) $ (for any $1< p < \infty $) to a limit functions $u_0, G_0$ in $\R^n$. By compactness argument the convergence is strong (also for any $1 < p < \infty $). Hence 
$$
\mathcal{I} (0_+, y^0, u, G)= 
\lim_{r_j \to 0}\mathcal{I} (r_js,y^0, u_{y^0}, G ) = \lim_{r_j \to 0} \mathcal{I} (s,0, (u_{y^0})_{r_j}, G_{r_j} ) =
\mathcal{I} (s,0, \max(0,u_0), G_0 ).
$$ 
Furthermore, the PDE for both  $G_0$ and $u_0$  is now  a constant coefficient PDE, $div (A_0 \nabla u_0) \geq 0$, and $div (A_0 \nabla G_0) = 0$.   Observe that under Lipschitz scaling, 
the right hand side, as well as all lower order derivatives vanish.
We may now consider a rotation and scaling of the coordinates $A_0^{1/2} x$, and define 
$$\tilde u_0 (x)= u_0(A_0^{1/2} x), \qquad
\tilde G_0 (x) = G_0(A_0^{1/2} x),
$$
which implies that $\Delta u_0 \geq 0$, and $\Delta G_0 =  0$ (in its support).  This in particular implies that 

$$\mathcal{I} (s,0, \max(0,\tilde u_0), \tilde G_0 ) = \hbox{Constant}.$$
By a strong form of monotonicity formula (for the Laplace operator) we must have $\{u_0 >0 \}$ is a half-space and $u_0$ is linear, unless $u_0 \equiv 0$, see \cite[Theorem 2.9]{PSU2012}.
This proves \eqref{eq:asymptotic2}.
\end{proof}

\section*{Acknowledgments}
This work was supported by the Science Committee of RA under the Research projects No 23RL-1A027 and No. 25RL-1A040. H. Shahgholian was supported in part by Swedish Research Council (grant no. 2021-03700).

\section*{Declarations}

\noindent {\bf  Data availability statement:} All data needed are contained in the manuscript.

\medskip
\noindent {\bf  Funding and/or Conflicts of interests/Competing interests:} The authors declare that there are no financial, competing or conflict of interests.

 \end{document}